\documentclass[12pt]{article}
\usepackage{amsmath,amsfonts,amsthm,amscd,upref,amstext}
\usepackage[dvips]{graphicx}
\usepackage{subfigure}
\usepackage{pb-diagram,pb-xy}
\usepackage[all]{xy}
\usepackage{fancyhdr}

\newtheorem{prop}{Proposition}[section]
\newtheorem{thm}[prop]{Theorem}
\newtheorem{cor}[prop]{Corollary}

\newtheorem{lem}[prop]{Lemma}
\newtheorem{nota}[prop]{Notation}
\newtheorem*{agra}{Acknowlegment}

\numberwithin{equation}{section}

\begin{document}

\title{On the Gorenstein Property of the Fiber Cone to Filtrations}
\author{P. H. Lima\thanks{Work partially supported by CNPq-Brazil 141973/2010-2 and by
Capes-Brazil.}\,\,\,\,and\,\,\,V. H. Jorge P\'erez
\thanks{Work supported by CNPq-Brazil - Grant
309033/2009-8, Procad-190/2007. {\it Key words}: fiber cone, associated graded ring, good and Hilbert filtration, reduction number.}}

\date{}
\maketitle

\begin{abstract} Let $(A, \mathfrak{m})$ be a Noetherian local ring and $\mathfrak{F}=(I_{n})_{n\geq 0}$ a filtration. In this paper, we study the Gorenstein properties of the fiber cone $F(\mathfrak{F})$, where $\mathfrak{F}$ is a Hilbert filtration. Suppose that $F(\mathfrak{F})$ and $G(\mathfrak{F})$ are Cohen-Macaulay. If in addition, the associated graded ring $G(\mathfrak{F})$ is Gorenstein; similarly to the $I$-adic case, we obtain a necessary and sufficient condition, in terms of lengths and minimal number of generators of ideals, for Gorensteiness of the fiber cone. Moreover, we find a description of the canonical module of $F(\mathfrak{F})$ and show that even in the Hilbert filtration case, the multiplicity of the canonical module of the fiber cone is upper bounded by multiplicity of the canonical modules of the associated graded ring.
\end{abstract}

\section{Introduction}
Let $(A, \mathfrak{m})$ be a Noetherian local ring and $A\supset I\supset I^{2}\supset...$ the adic-filtration. Then we have important graded algebras such as $R(I):=\oplus_{n\geq 0}I^{n}t^{n}$, the associated graded ring $G(I):=\oplus_{n\geq 0}I^{n}/I^{n+1}$ and the fiber cone $F(I):=\oplus_{n\geq 0}I^{n}/\mathfrak{m}I^{n}=R(I)/ \mathfrak{m}R(I)$. These algebras may be naturally generalized for the general filtration case of ideals, that is, $\mathfrak{F}: A \supset I_{1}\supset I_{2}\supset...$ .  We may ask when these algebras are Gorenstein or Cohen-Macaulay and other properties. This answer is given in many papers (\cite{HRZ}, \cite{GN}, \cite{HZ}, \cite{TVZ},
\cite{I}, \cite{HKU1}, \cite{HKU2}, \cite{JPV}). In \cite{HRZ}, it is determined, under certain classes of ideals, the exponent $n\geq 0$ such that $R(I^{n})$ and $G(I^{n})$ are Gorenstein. Goto and Nishida \cite{GN} investigate the Cohen-Macaulay and Gorenstein properties of symbolic Rees algebras for one-dimensional prime ideals in Cohen-Macaulay local rings, and then they generalize to Rees algebras $R(\mathfrak{F})$ and graded rings $G(\mathfrak{F})$ associated to general filtrations of ideals in arbitrary Noetherian local rings. Moreover, in \cite{HZ}, Zarzuela and Hoa give some consequences which generalize well known results in the $I$-adic case. For instance, they show that, under certain assumptions, the Gorensteinness of $R(\mathfrak{F})$ implies the Gorensteinness of $G(\mathfrak{F})$. Trung, Vi\^{e}t and Zarzuela, in \cite{TVZ}, extend the Ikeda's result \cite{I} to Rees algebra of filtrations. In the article \cite{HKU1}, Heinzer, Kim and Ulrich give a criterion, in terms of ideals, for the Gorensteinness of $G(I)$. Next they generalized this result for Hilbert filtrations \cite{HKU2}. Jayanthan, Puthenpurakal and Verma, in \cite{JPV}, study the Gorenstein property of $F(I)$ of $\mathfrak{m}$-primary ideals and the ones of almost minimal multiplicity.

In this article, we are interested on the Gorenstein property of the fiber cone to filtration. In a recent article, for the $I$-adic case, Jayanthan and Nanduri \cite{JN}, find descriptions, in terms of ideals, for the canonical module of the fiber cone. They also show that if both the associated graded ring and fiber cone are Cohen-Macaulay, then the regularities of the their canonical modules are equal. Furthermore, when $F(I)$ is Cohen-Macaulay and $G(I)$ is Gorenstein, they give a criterion to say when the fiber cone is Gorenstein. The main goal in this paper is to give an analogous theory on the Gorenstein property of the fiber cone for general filtration. We show that the results obtained in \cite{JN} may be generalized for the Hilbert filtration.

This paper is divided into two parts. In section 2 we introduce the basic concepts about good and Hilbert filtration, reduction of filtration, canonical module and regularity. In section 3 we extend the Theorem 15 and Theorem 33 in \cite{CZ} for the good filtration case. Moreover, we find for the Hilbert filtration case, Gorenstein properties of the fiber cone and descriptions of its canonical module in terms of ideals. We also see that the multiplicity of the canonical module of the fiber cone is less than or equal to the multiplicity of the canonical module of the associated graded ring even for Hilbert filtration and that if in addition, $G(\mathfrak{F})$ is Gorenstein, the equality is true if only if $I_{1}=\mathfrak{m}$. The main results in this section generalize the results in the section 3 of \cite{JN} to Hilbert filtration.

\section{Preliminaries}
A sequence $\mathfrak{F}=(I_{n})_{n\geq 0}$ of ideals of $A$ is called a \emph{filtration} of $A$ if $I_{0}=A\supset I_{1}\supseteq I_{2}\supseteq I_{3}\supseteq ...,$ $I_{1}\neq A$ and $I_{i}I_{j}\subseteq I_{i+j}$ for all $i,j\geq 0$.

Let $I$ be and ideal of $A$. $\mathfrak{F}$ is called an $I$-\emph{good filtration} if $II_{i}\subseteq I_{i+1}$ for all $i\geq 0$ and $I_{n+1}=II_{n}$ for all $n\gg 0$. $\mathfrak{F}$ is called a \emph{good filtration} if it is an $I$-good filtration for some ideal $I$ of $A$. $\mathfrak{F}$ is a good-filtration if and only if it is a $I_{1}$-good filtration.

Given any filtration $\mathfrak{F}$, we can construct the following graded rings
$$
\begin{array}{c}
  R(\mathfrak{F})=A \oplus I_{1}t \oplus I_{2}t^{2}\oplus ... \  , \ R^{'}(\mathfrak{F})=\bigoplus_{n\in \mathbb{Z}} I_{n}t^{n} \subseteq A[t,t^{-1}], \vspace{0.3cm} \\
  G(\mathfrak{F})= R^{'}(\mathfrak{F}) / (t^{-1})R^{'}(\mathfrak{F})   =A/I_{1}\oplus I_{1}/I_{2} \oplus I_{2}/I_{3} \oplus ...
\end{array}
$$
We call $R(\mathfrak{F})$ the \emph{Rees algebra} of $\mathfrak{F}$, $R^{'}(\mathfrak{F})$ the \emph{Extended Rees algebra} of $\mathfrak{F}$ and $G(\mathfrak{F})$ the \emph{associated graded ring} of $\mathfrak{F}$.
We also denote the \emph{irrelevant ideal} of $G(\mathfrak{F})$ by $G(\mathfrak{F})_{+}=\oplus_{n\geq 1}I_{n}/I_{n+1}$. If $\mathfrak{F}$ is an $I$-adic filtration, i.e, $\mathfrak{F}=(I^{n})_{n\geq 0}$ for some ideal $I$, we denote $R(\mathfrak{F})$ and $G(\mathfrak{F})$ by $R(I)$ and $G(I)$, respectively. A filtration $\mathfrak{F}$ is called \emph{Noetherian} if $R(\mathfrak{F})$ is a Noetherian ring. Noetherian filtration satisfies $\cap_{n\geq 0}I_{n}=0$. By adapting the proof of \cite[Theorem 15.7]{M}, one can prove that if $\mathfrak{F}$ is Noetherian, then $\dim \hspace{0.01cm} G(\mathfrak{F})= \dim \hspace{0.01cm} A$. We often denote $\text{depth}  \hspace{0.01cm} G(\mathfrak{F})$ for $\text{depth}_{\mathfrak{M}}  G(\mathfrak{F})$ where $\mathfrak{M}=\mathfrak{m}/I_{1}\oplus I_{1}/I_{2}\oplus I_{2}/I_{3}...$ is the unique homogeneous maximal ideal of $G(\mathfrak{F})$.

A filtration $\mathfrak{F}$ which is $I_{1}$-good and $I_{1}$ is $\mathfrak{m}$-primary, is called \emph{Hilbert filtration}. If $\mathfrak{F}$ is a Hilbert filtration, then $G(\mathfrak{F})$ is a finite $G(I_{1})$-module.

For any filtration $\mathfrak{F}= (I_{n})$ and any ideal $J$ of $A$, one denote $\mathfrak{F}/J$ to be the filtration $((I_{n}+J)/J)_{n}$ in the ring $A/J$. Besides, if $\mathfrak{F}$ is Noetherian (Hilbert) then $\mathfrak{F}/J$ is Noetherian (Hilbert).

A \emph{reduction} of a filtration $\mathfrak{F}$ is an ideal $J\subseteq I_{1}$ such that $JI_{n}=I_{n+1}$ for $n\gg 0$. We also know that $J\subseteq I_{1}$ is a reduction of $\mathfrak{F}$ if and only if $R(\mathfrak{F})$ is a finite $R(I_{1})$-module. By \cite[Theorem III.3.1.1 and Corollary III.3.1.4]{B}, $R(\mathfrak{F})$ is a finite $R(I_{1})$-module if and only if there exists an integer $k$ such that $I_{n}\subseteq (I_{1})^{n-k}$ for all $n$. A \emph{minimal reduction} of $\mathfrak{F}$ is a reduction of $\mathfrak{F}$ minimal with respect to containment. By \cite[Proposition 2.6]{HZ}, minimal reductions of $\mathfrak{F}$ do exist. A good filtration $\mathfrak{F}$ is called \emph{equimultiple} if $I_{1}$ is equimultiple, i.e, $s(I_{1})=  \text{ht} \ I_{1}$. Hilbert filtrations are examples of equimultiple filtrations.

If $J$ is a reduction of the $I$-adic filtration, we say simply that $J$ is a reduction of $I$. By \cite{NR},  minimal reduction of ideals always exist. If $R(\mathfrak{F})$ is a $R(I_{1})$-module, then $J$ is a reduction of $\mathfrak{F}$ if and only if $J$ is a reduction of $I_{1}$. Thus minimal reductions of good filtration always exist and are generated by $\dim A$ elements if $A/\mathfrak{m}$ is infinite. For minimal reduction $J$ of $\mathfrak{F}$ we set $r_{J}(\mathfrak{F})= \min\{n\in \mathbb{Z} \ | \ I_{n}= JI_{n-1}\}.$ The \emph{reduction number} of $\mathfrak{F}$ is defined as $r(\mathfrak{F})= \min \{ r_{J}(\mathfrak{F}) \ | \ J  \text{ is minimal reduction of} \ \mathfrak{F}\}$.
Let $\mathfrak{F}$ be a Noetherian filtration. For any element $x\in I_{1}$, we let $x^{*}$ denote the image of $x$ in $G(\mathfrak{F})_{1}=I_{1}/I_{2}$ and $x^{o}$ denote the image of $x$ in $F(\mathfrak{F})_{1}=I_{1}/\mathfrak{m}I_{1}$. If $x^{*}$ is a regular element of $G(\mathfrak{F})$, then $x$ is a regular element of $A$ and by \cite[Lemma 3.4]{HZ}, $G(\mathfrak{F}/(x))\cong G(\mathfrak{F})/(x^{*})$.

An element $x\in I_{1}$ is called \emph{superficial} for $\mathfrak{F}$ if there exists an integer $c$ such that $(I_{n+1}: \ x)\cap I_{c}=I_{n}$ for all $n\geq c$. By \cite[Remark 2.10]{HZ}, an element $x$ is superficial for $\mathfrak{F}$ if and only if $(0:_{G(\mathfrak{F})} \ x^{*})_{n}=0$ for all $n$ sufficiently large. If $\text{grade} \ I_{1} \geq 1$ and $x$ is superficial for $\mathfrak{F}$, then $x$ is a regular element of $A$. To see that, let suppose that $x$ is not a zero-divisor. Thus if $ux=0$, then $(I_{1})^{c}u\subseteq \cap_{n}((I_{n}: \ x)\cap I_{c})=\cap_{n}I_{n}=0$. Hence $u=0$.

A sequence $x_{1},...,x_{k}$ is called a \emph{superficial sequence} for $\mathfrak{F}$ if $x_{1}$ is superficial for $\mathfrak{F}$ and $x_{i}$ is superficial  for $\mathfrak{F}/(x_{1},...,x_{i-1})$ where $2\leq i \leq k$.

Let $f_{1},...,f_{r}$ be a sequence of homogeneous elements of a noetherian graded algebra $S=\oplus_{n\geq 0}S_{n}$ over a local ring $S_{0}$. It is called \emph{filter-regular sequence} of $S$ if $f_{i}\not\in \mathfrak{p}$ for all primes $\mathfrak{p}\in \text{Ass}(S/(f_{1},...,f_{i-1}))$ such that $S_{+}\not\subseteq \mathfrak{p}$, $i=1,...,r$.

Let $(A, \mathfrak{m})$ be a local ring and $\mathfrak{F}$ a good filtration. Then $v_{1},...,v_{t}\in I_{1}$ are \emph{analytically independent} in $\mathfrak{F}$ if and only if, whenever $h\in \mathbb{N}$ and $f\in A[X_{1},...,X_{t}]$ (the ring of polynomials over $A$ in $t$ indeterminates) is a homogeneous polynomial of degree $h$ such that $f(v_{1},...,v_{t})\in I_{h}\mathfrak{m}$, then all coefficients of $f$ lie in $\mathfrak{m}$. Moreover, if $v_{1},...,v_{t}\in I_{1}$ are analytically independent in $\mathfrak{F}$ and $J=(v_{1},...,v_{t})$, then $J^{h}\cap I_{h}\mathfrak{m}=J^{h}\mathfrak{m}$ for all $h\in \mathbb{N}$.

Now let define the analytic spread of a filtration $\mathfrak{F}$. Firstly we define
$$
F(\mathfrak{F})=\bigoplus_{n \geq 0} \frac{ I_{n} }  { \mathfrak{m}I_{n} }= A/\mathfrak{m} \oplus I_{1}/I_{2} \oplus ...,
$$
which is called \emph{Fiber Cone} of $\mathfrak{F}$.
The number $s=s(\mathfrak{F})= \dim F(\mathfrak{F}) =\dim R(\mathfrak{F})/\mathfrak{m}R(\mathfrak{F})$ \cite{HZ} is said to be the \emph{analytic spread} of $\mathfrak{F}$. Thus, when $\mathfrak{F}$ is the $I$-adic filtration, the analytic spread $s(I)$ equals $s(\mathfrak{F})$. Rees introduced the notion of basic reductions of Noetherian filtrations and he showe in \cite[Theorem 6.12]{R} that $s(\mathfrak{F})$ equals the minimal number of generators of any minimal reduction. By \cite[Lemma 2.7]{HZ}, $s(\mathfrak{F})=\dim G(\mathfrak{F})/\mathfrak{m}G(\mathfrak{F})$ and by \cite[Lemma 2.8]{HZ}, $s(\mathfrak{F})=s(I_{1})$.

If $(S,\mathfrak{m})$ is a Cohen-Macaulay $^{*}$local ring of $^{*}$dimension $d$, a finite graded $S$-module $\omega_{S}$ is a $^{*}$\emph{canonical module} of $S$ if there exists homogeneous isomorphisms
$$
^{*}\text{Ext}_{S}^{  i}  (S/\mathfrak{m}, \omega_{S}) \cong  \left\{  \begin{array}{cc}
                                                                   0 & \text{for} \ i \neq d \\
                                                                   S/\mathfrak{m} & \text{for} \ i = d
                                                                 \end{array} \right.
$$
Moreover, $\omega_{S}$ is a Cohen-Macaulay graded $A$-module with $\text{depth} \ \omega_{S}= \dim S$ and finite $^{*}$injective dimension \cite{HIO}. It is also known that $S$ is Gorenstein if and only if $\omega_{S} \cong S(a)$ for some integer $a\in \mathbb{Z}$  \cite{BH}.

Let $S=\bigoplus_{n\geq 0}S_{n}$ be a finitely generated standard graded ring over a Noetherian commutative ring $S_{0}$. For any graded $S$-module $M$, we denote by $M_{n}$, the homogeneous part of degre $\emph{n}$ of $M$ and one define
$$
a(M):= \left\{ \begin{array}{ll}
           \max \{n \ | \ M_{n}\neq 0 \} & \text{if}  \  M \neq 0 \\
           -\infty & \text{if} \ M = 0
         \end{array}
         \right
 .$$
Let $S_{+}$ be the ideal generated by the homogeneous elements of positive degree of $S$. For $i\geq 0$, set $$a_{i}(S):=a(H_{S_{+}}^{i}(S)),$$ where $H_{S_{+}}^{i}(.)$ denotes the \emph{i}-th local cohomology functor with respect to the ideal $S_{+}$. More generally, for $i\geq 0$ and any graded $S$-module $M$, set
$$a_{i}(M):=a(H_{S_{+}}^{i}(M)),$$ where $H_{S_{+}}^{i}(M)$ denotes the \emph{i}-th local cohomology module of M with respect to the irrelevant ideal $S_{+}$. The \emph{Castelnuovo-Mumford regularity} (or simply \emph{regularity}) of $M$ is defined as the number $$\text{reg}(M):= \max\{a_{i}(M)+i \ | \ i\geq 0\}.$$
When $M=S$, the regularity $\text{reg} \ S$ is an important invariant of the graded ring $S$, \cite{EG} and \cite{O}.

\section{Gorenstein Property of the Fiber Cone to Filtrations}
Throughout this section, $(A,\mathfrak{m})$ is a Noetherian local ring of dimension $d$ with an infinite residue field and $\mathfrak{F}$ a good filtration. Moreover we always suppose that $F(\mathfrak{F})$ and $G(\mathfrak{F})$ are Cohen-Macaulay. If the dimension of the fiber cone is greater than zero, by \cite[Lemma 3.1]{CZ0}, $I_{n+1}\subseteq\mathfrak{m}I_{n}$ for all $n\geq 0$ and then $F(\mathfrak{F})=G(\mathfrak{F})/\mathfrak{m}G(\mathfrak{F})$. In this article we assume this fact if $A$ has dimension zero. It is required in the most of the results of this article such as the Proposition below. In \cite{JN}, Jayanthan and Nanduri obtain, for a the $I$-adic filtration case, an expression for the canonical module of the fiber cone. We verify this result remains true for the canonical module of the fiber cone if the filtration $\mathfrak{F}$ is such that $I_{1}$ is $\mathfrak{m}$-primary (Hilbert filtration).

\begin{prop}\label{prop3.1}
Let $(A,\mathfrak{m})$ be a Noetherian local ring and $\mathfrak{F}$ a Hilbert filtration such that the associated graded ring $G(\mathfrak{F})$ is Cohen-Macaulay. Let $\omega_{G(\mathfrak{F})}=\oplus_{n\in \mathbb{Z}}\omega_{n}$ and $\omega_{F(\mathfrak{F})}$ be the canonical modules of $G(\mathfrak{F})$ and $F(\mathfrak{F})$ respectively. Then

\begin{description}
           \item[(1)] $\omega_{F(\mathfrak{F})}\cong \oplus_{n\in \mathbb{Z}}(0:_{\omega_{n}} \mathfrak{m});$

           \item[(2)] $a(F(\mathfrak{F}))=a(G(\mathfrak{F}))=r_{J}(\mathfrak{F})-d$, where $r:=r_{J}(\mathfrak{F})$ is the reduction number of $\mathfrak{F}$ with respect to $J$ to any reduction minimal $J \subseteq I_{1}$;

           \item[(3)] If $\dim A> 0$, $a(F(\mathfrak{F}^{(k)}))=[\frac{a(F(\mathfrak{F}))}{k}]=[\frac{r-d}{k}]$ for any $k\in \mathbb{N}$.

           \item[(4)] if $G(\mathfrak{F})$ is Gorenstein, $$\omega_{F(\mathfrak{F})} \cong \bigoplus_{n\in \mathbb{Z}}\frac{(I_{n+r-d+1}:\mathfrak{m})\cap I_{n+r-d}}{I_{n+r-d+1}} .$$
\end{description}

\end{prop}

\begin{proof}
\textbf{(1)} By adapting the proof of \cite[Theorem 15.7]{M}, one can prove $\dim G(\mathfrak{F})=d$. Since $I_{1}$ is an $\mathfrak{m}$-primary ideal, $\dim F(I_{1})=d$. By \cite[Lemma 2.8]{HZ}, $\dim F(\mathfrak{F})= \dim F(I_{1})$. Thus $\dim G(\mathfrak{F}) = \dim F(\mathfrak{F})=d$. By hypothesis, $G(\mathfrak{F})$ is Cohen-Macaulay, so by using \cite[Corollary 36.14]{HIO}, we have $$\omega_{F(\mathfrak{F})} \cong \text{Hom}_{G(\mathfrak{F})}(F(\mathfrak{F}), \omega_{G(\mathfrak{F})}).$$ It is easy to show that
$$\text{Hom}_{G(\mathfrak{F})}(F(\mathfrak{F}), \omega_{G(\mathfrak{F})})\cong (0:_{\omega_{G(\mathfrak{F})}}\mathfrak{m}G(\mathfrak{F}))=(0:_{\omega_{G(\mathfrak{F})}}\mathfrak{m})=\oplus_{n \in \mathbb{Z}}(0:_{\omega_{n}}\mathfrak{m}).$$
The result follows.

\textbf{(2)} By \cite[Definition 3.6.13]{BH}, $a(F(\mathfrak{F}))=-\min \{n | [\omega_{F(\mathfrak{F})}]_{n}\neq 0 \}$. Since $A/I_{1}$ is Artinian and $\omega_{n}$ is a finite $A/I_{1}$-module, it follows that $\omega_{n}$ is Artinian. By \cite[Theorem A.33]{ILL}, $\omega_{n}$ is a $\mathfrak{m}/I_{1}$- torsion and by \cite[Example A.9]{ILL}, $$(0:_{\omega_{n}}\mathfrak{m})=(0:_{\omega_{n}}\mathfrak{m}/I_{1})\subseteq \omega_{n}$$ is an essential extension. Thus
\begin{equation}\label{extessencial}
(0:_{\omega_{n}}\mathfrak{m})\neq 0 \ \text{if and only if } \omega_{n} \neq 0.
\end{equation}
Therefore by \textbf{(1)}, $a(F(\mathfrak{F}))=a(G(\mathfrak{F}))$.

Through \cite[Theorem 4.5.7]{BH}, one has $A$ is Cohen-Macaulay. Note also that $\mathfrak{F}$ is equimultiple since $I_{1}$ is $\mathfrak{m}$-primary. As $\text{grade} \hspace{0.05cm} G(\mathfrak{F})_{+}=d$, the hypothesis of \cite[Proposition 3.6]{HZ} are satisfied and then the second equality follows.

\textbf{(3)} By \cite[Corollary 4.6]{HZ}, $G(\mathfrak{F}^{(k)})$ is Cohen-Macaulay. As $\dim A/I_{1}=0<d$, by \cite[Theorem 4.2]{HZ}, $a(G(\mathfrak{F}^{(k)}))=[\frac{a(G(\mathfrak{F}))}{k}]$. Applying \textbf{(2)} to Veronesean filtration $\mathfrak{F}^{(k)}$, we have
$$a(F(\mathfrak{F}^{(k)}))=a(G(\mathfrak{F}^{(k)}))=[\frac{a(G(\mathfrak{F}))}{k}]=[\frac{a(F(\mathfrak{F}))}{k}]=[\frac{r-d}{k}].$$

\textbf{(4)} Let suppose that $G(\mathfrak{F})$ is Gorenstein. By \cite[36.13]{HIO}, $$\omega_{G(\mathfrak{F})} \cong G(\mathfrak{F})(a(G(\mathfrak{F})))=G(\mathfrak{F})(r-d).$$ That is $\omega_{n}= I_{n+r-d}/I_{n+r-d+1}$ for any $n$. By using the last fact and \textbf{(1)}, we have
$$\omega_{F(\mathfrak{F})}\cong \oplus_{n\in \mathbb{Z}}(0:_{\omega_{n}} \mathfrak{m})=\bigoplus_{n\in \mathbb{Z}}\frac{(I_{n+r-d+1}:\mathfrak{m})\cap I_{n+r-d}}{I_{n+r-d+1}}$$
\begin{equation}  \label{n>d-r}  \hspace{6,018cm} \hspace{-0.1cm} \left( =\bigoplus_{n\geq d-r}\frac{(I_{n+r-d+1}:\mathfrak{m})\cap I_{n+r-d}}{I_{n+r-d+1}} \right).\end{equation}

\end{proof}

\begin{cor}
 Suppose that $A$ has dimension $\geq 2$. If $R(\mathfrak{F})$ is Gorenstein,
$$\omega_{F(\mathfrak{F})}\cong \oplus_{n\in \mathbb{Z}}(0:_{\omega_{n}} \mathfrak{m})=\bigoplus_{n\in \mathbb{Z}}\frac{(I_{n-1}:\mathfrak{m})\cap I_{n-2}}{I_{n-1}}.$$
That is, the canonical module $\omega_{F(\mathfrak{F})}$ does not depend on the dimension $d$ and the number reduction $r$.
\end{cor}
\begin{proof}
Use \cite[Corollary 3.11]{HZ}.
\end{proof}
\begin{nota}

By \emph{\cite[\text{Proposition 2.2}]{JV}}, we get a minimal reduction $J=(x_{1},...,x_{d})$ of $\mathfrak{F}$ such that $x_{1}^{*},...,x_{d}^{*}\in G(\mathfrak{F})$ and $x_{1}^{o},...,x_{d}^{o}\in F(\mathfrak{F})$ are superficial sequence. Denote $J_{i}=(x_{1},...,x_{i})$, $J_{0}=0$ and $J^{*}=(x_{1}^{*},...,x_{d}^{*})$, $J^{o}=(x_{1}^{o},...,x_{d}^{o})$.

\end{nota}

\begin{lem}\label{lemgor}
Suppose $G(\mathfrak{F})$ and $F(\mathfrak{F})$ are Cohen-Macaulay rings. Then $G(\mathfrak{F})$ is Gorenstein if and only if $G(\mathfrak{F}/J_{i})$ is Gorenstein and $F(\mathfrak{F})$ is Gorenstein if and only if $F(\mathfrak{F}/J_{i})$ is Gorenstein.
\end{lem}

\begin{proof}
Let $J=(x_{1},...,x_{d})$ be a minimal reduction of $\mathfrak{F}$ such that $x_{1}^{*},...,x_{d}^{*}\in G(\mathfrak{F})$ and $x_{1}^{o},...,x_{d}^{o}\in F(\mathfrak{F})$ are superficial sequence. Since $G(\mathfrak{F})$ and $F(\mathfrak{F})$ are Cohen-Macaulay rings, these sequences are also regular sequences (the argument is similar to \cite[Theorem 8]{P}). Then for any $i$ such that $1\leq i\leq d$, $G(\mathfrak{F})/(x_{1}^{*},...,x_{i}^{*})\cong G(\mathfrak{F}/(x_{1},...,x_{i}))$ and $F(\mathfrak{F})/(x_{1}^{*},...,x_{i}^{*})\cong F(\mathfrak{F}/(x_{1},...,x_{i}))$ (see \cite[Lemma 4.5]{LP}). By \cite[Proposition 3.1.19]{BH}, the result follows.
\end{proof}

\begin{lem}\label{lemz}
Suppose $\ell=\dim F(\mathfrak{F})=1$ and let $J=(a)$ be a minimal reduction of $\mathfrak{F}$. Assume that $I_{1}$ contains a regular element.
Then for all $n$ and $0\leq i \leq n-1$, we have
$$\lambda (I_{n}/(\mathfrak{m}I_{n}+ a^{n-i}I_{i}))= \mu(I_{n})-\mu(I_{i}) + \lambda ( (a^{n-i}I_{i} \cap \mathfrak{m}I_{n})/a^{n-i}\mathfrak{m}I_{i} ).$$
Particularly,
$$\lambda (I_{n}/(\mathfrak{m}I_{n}+ aI_{n-1}))= \mu(I_{n})-\mu(I_{n-1}) + \lambda ( (aI_{n-1} \cap \mathfrak{m}I_{n})/a\mathfrak{m}I_{n-1} ).$$

\end{lem}

\begin{proof}
The assumption on $I_{1}$ gives that $a$ is regular.
The exact sequence
$$
\begin{array}{ccccccccc}
    0 & \rightarrow & (a^{n})/(\mathfrak{m}I_{n}\cap (a^{n})) & \rightarrow & I_{n}/\mathfrak{m}I_{n}  & \rightarrow & I_{n}/(\mathfrak{m}I_{n}+(a^{n}))  & \rightarrow & 0
  \end{array}
$$
and the equality $\mathfrak{m}I_{n}\cap (a^{n})=a^{n}\mathfrak{m}$ ($J$ is analytically independent in $\mathfrak{F}$) give $\lambda (I_{n}/(\mathfrak{m}I_{n}+(a^{n})))=\mu(I_{n})-\mu(J^{n})=\mu(I_{n})-1$.
For $1\leq i \leq n-1$, we have the following exact sequences
$$
\begin{array}{ccccccccc}
    0 & \rightarrow & a^{n-i}I_{i}/(a^{n-i}I_{i}\cap \mathfrak{m}I_{n}) & \rightarrow & I_{n}/\mathfrak{m}I_{n}  & \rightarrow & I_{n}/(\mathfrak{m}I_{n}+a^{n-i}I_{i})  & \rightarrow & 0,
  \end{array}
$$
$$
\begin{array}{lllllllll}
    0\hspace{-0.3cm} & \rightarrow & \hspace{-0.3cm} (a^{n-i}I_{i}\cap \mathfrak{m}I_{n})/a^{n-i}\mathfrak{m}I_{i}  \hspace{-0.3cm}& \rightarrow \hspace{-0.3cm}& a^{n-i}I_{i}/ a^{n-i}\mathfrak{m}I_{i} \hspace{-0.3cm} & \rightarrow & \hspace{-0.3cm} a^{n-i}I_{i}/(a^{n-i}I_{i}  \cap \mathfrak{m}I_{n})  &\hspace{-0.3cm} \rightarrow & \hspace{-0.3cm} 0.
  \end{array}
$$
We also have the isomorphism
$$a^{n-i}I_{i}/a^{n-i}\mathfrak{m}I_{i} \cong  I_{i}/\mathfrak{m}I_{i} ,$$
since $a$ is a regular element. The result follows due to the additivity of $\lambda$.
\end{proof}

\begin{nota}

If $a_{1},...,a_{k}\in A$ and $\mathfrak{F}$ is a good filtration, we denote by $I_{n}^{(k)}=(I_{n}+(a_{1},...,a_{k}))/(a_{1},...,a_{k})$
the $n$-th term of the filtration $\mathfrak{F}/(a_{1},...,a_{k})$ and the quotient ring $A/(a_{1},...,a_{k-1})$ by $A_{k-1}$ .

\end{nota}

\begin{lem}\label{lemz31}
If $a_{1},...,a_{k}\in I_{1}$ are such that $a_{1}^{*},...,a_{k}^{*}\in G(\mathfrak{F})$ and $a_{1}^{o},...,a_{k}^{o}\in F(\mathfrak{F})$ are regular sequence, then
$$
\mu(I_{n}^{(k)})=\sum_{i=0}^{k}(-1)^{i}\binom{k}{i}\mu(I_{n-i}).
$$
\end{lem}

\begin{proof}
The proof is done by using induction on $k$. It is easy to obtain $(a_{1})\cap I_{n}=a_{1}I_{n-1}$ for all $n \geq 0$. If $k=1$, by the isomorphism theorem,
$$
\mu(I_{n}^{(1)})=\lambda((I_{n}+(a_{1}))/(\mathfrak{m}I_{n}+(a_{1})) =\lambda(I_{n}/(\mathfrak{m}I_{n}+a_{1}I_{n-1})).
$$
Let consider the following natural exact sequence
$$
0 \rightarrow a_{1}I_{n-1}/(a_{1}I_{n-1}\cap \mathfrak{m}I_{n}) \rightarrow I_{n}/\mathfrak{m}I_{n} \rightarrow I_{n}/(\mathfrak{m}I_{n}+a_{1}I_{n-1}) \rightarrow 0.
$$
From \cite[Proposition 4.4]{C}, we have $a_{1}I_{n-1}\cap \mathfrak{m}I_{n}=a_{1}\mathfrak{m}I_{n-1}$. Moreover by \cite[Proposition 3.5]{HM}, $a_{1}$ is a regular element. Then $a_{1}I_{n-1}/a_{1}\mathfrak{m}I_{n-1}\cong I_{n-1}/\mathfrak{m}I_{n-1}$. Thus by exact sequence above, $\mu(I_{n}^{(1)})=\mu(I_{n})-\mu(I_{n-1})$.

Now suppose $k>1$. Hence $\overline{a}_{k}\in A_{k-1}$, $(\overline{a}_{k})^{*}\in G(\mathfrak{F}/(a_{1},...,a_{k-1}))$ and $(\overline{a}_{k})^{o}\in F(\mathfrak{F}/(a_{1},...,a_{k-1}))$ are regular elements. We know that the equality $\mu(I_{n}^{(k)})=\mu(I_{n}^{(k-1)})-\mu(I_{n-1}^{(k-1)})$ is true for $k=1$. Then, by induction on $k$,
$$
\begin{array}{lll}
  \mu(I_{n}^{(k)}) & = & \sum_{i=0}^{k-1}(-1)^{i}\binom{k-1}{i}\mu(I_{n-i}) - \sum_{i=0}^{k-1}(-1)^{i}\binom{k-1}{i}\mu(I_{n-1-i}) \vspace{0.3cm} \\
   & = &  \sum_{i=0}^{k}(-1)^{i} \left( \binom{k-1}{i} + \binom{k-1}{i-1}\right) \mu(I_{n-i}) \vspace{0.3cm} \\
   & = & \sum_{i=0}^{k}(-1)^{i}\binom{k}{i}\mu(I_{n-i})
\end{array}
$$
\end{proof}

The next lemma gives several characterizations of the Cohen-Macaulay property of the fiber cone of dimension one to good filtration.

\begin{lem}\label{zteo15}
Let $(A, \mathfrak{m})$ be a Noetherian Local ring and let $\mathfrak{F}$ a good filtration such that $I_{1}$ contains a regular element, $\dim F(\mathfrak{F})=1$ and $r$ is its reduction number. Then we have the following equivalent conditions:
\begin{description}
  \item[(1)] $F(\mathfrak{F})$ is Cohen-Macaulay;

  \item[(2)] $(0:_{F(\mathfrak{F})} a^{o})=0$ for any $J=(a) \subseteq I_{1}$ minimal reduction of $\mathfrak{F}$;

  \item[(3)] For any minimal reduction $(a)$ of $\mathfrak{F}$, $I_{n}\cap (\mathfrak{m}I_{n+1}: a)=\mathfrak{m}I_{n},$ $1\leq n \leq r-1;$

  \item[(4)] For any minimal reduction $(a)$ of $\mathfrak{F}$, $aI_{n}\cap \mathfrak{m}I_{n+1}=a\mathfrak{m}I_{n},$ $1\leq n \leq r-1;$

  \item[(5)] For any minimal reduction $(a)$ of $\mathfrak{F}$, $\lambda (I_{n}/(\mathfrak{m}I_{n}+aI_{n-1}) )=\mu(I_{n})-\mu(I_{n-1})$, $1\leq n \leq r.$
\end{description}

\end{lem}

\begin{proof}
The sentences \textbf{(1)} $\Leftrightarrow$ \textbf{(2)}, \textbf{(2)} $\Leftrightarrow$ \textbf{(3)} and \textbf{(3)} $\Leftrightarrow$ \textbf{(4)} follow by using the fact that $a$ is regular in $A$, while \textbf{(4)} $\Leftrightarrow$ \textbf{(5)} is provided by Lemma \ref{lemz}.

\end{proof}

In \cite{CZ3}, Cortadellas and Zarzuela show under some assumptions that the regularity of the fiber cone is equal to reduction number of a minimal reduction $J$ of $I$. We verify the same result for equimultiple good filtration.

\begin{prop}\label{regf}
Let $(A, \mathfrak{m})$ be a Noetherian local ring and let $\mathfrak{F}$ an equimultiple good filtration. Denote $\ell=\dim F(\mathfrak{F})$ and assume $\text{grade} \hspace{0.05cm} I_{1}=\ell$, $\text{grade} \hspace{0.05cm} G(\mathfrak{F})_{+}\geq \ell-1$ and $\text{depth} \hspace{0.01cm} F(\mathfrak{F}) \geq \ell-1$. Let $J\subseteq I_{1}$ be a minimal reduction of $\mathfrak{F}$ and denote $r=r_{J}(\mathfrak{F})$, the reduction number of $\mathfrak{F}$ with respect to $J$. Then we have the following:
\begin{description}
  \item[(1)] $\emph{reg} \ F(\mathfrak{F})=r;$

  \item[(2)] $F(\mathfrak{F})$ is a Cohen-Macaulay ring if and only if
  $$\lambda(I_{n}/(\mathfrak{m}I_{n}+JI_{n-1}))=\sum_{i=0}^{\ell}(-1)^{i}\binom{\ell}{i} \mu(I_{n-i}),$$ for $1\leq n \leq r(\mathfrak{F})$.

\end{description}

\end{prop}

\begin{proof}
If $\ell=1$ the results follow by \cite[Lemma 4.2]{LP} and Lemma \ref{zteo15}. Suppose $\ell > 1$. Since $\text{grade} \ I_{1}=\ell$ we may find a minimal reduction $J=(a_{1},...,a_{\ell})$ of $\mathfrak{F}$ such that $a_{1},...,a_{\ell}\in A$ a regular sequence, $a_{1}^{*},...,a_{l-1}^{*}\in G(\mathfrak{F})$ is a regular sequence and $a_{1}^{o},...,a_{l-1}^{o}\in F(\mathfrak{F})$ is a regular sequence. As $I_{1}/(a_{1},...,a_{\ell-1})\subset A/(a_{1},...,a_{\ell-1})$ contains a regular element and besides $\dim F(\mathfrak{F}/(a_{1},...,a_{\ell-1}))=1$ (see \cite[Proposition 2.5]{JV}) and $r_{J}(\mathfrak{F})=r_{J/J_{\ell-1}}(\mathfrak{F}/J_{\ell-1})$, we have that
$$\text{reg} \ F(\mathfrak{F})= \text{reg} \ F(\mathfrak{F})/(a_{1}^{o},...,a_{l-1}^{o}) = \text{reg} \ F(\mathfrak{F}/J_{\ell-1})=r.$$ We have done \text{\textbf{(1)}}.

\text{\textbf{(2)}} By \cite[Theorem 2.1.3]{BH} and the fact that
$$F(\mathfrak{F}/(a_{1},...,a_{\ell-1}))\cong F(\mathfrak{F})/(a_{1}^{o},...,a_{l-1}^{o}),$$
we have that
$F(\mathfrak{F})$ is Cohen-Macaulay if, and only if $F(\mathfrak{F}/(a_{1},...,a_{\ell-1}))$ is Cohen-Macaulay. Let denote by $\overline{I}_{n}=(I_{n}+J_{\ell-1})/J_{\ell-1}$ the $n$-th term of $\mathfrak{F}/(a_{1},...,a_{\ell-1})$ and $\mathfrak{m}_{\ell-1}=\mathfrak{m}/J_{\ell-1}\subset A/J_{\ell-1}$. By Theorem \ref{zteo15}, the argument above is equivalent to
$$
\lambda(\overline{I}_{n}/\mathfrak{m}_{\ell-1}\overline{I}_{n}+a_{\ell}\overline{I}_{n-1})=\mu(\overline{I}_{n-1})-\mu(\overline{I}_{n-1}),
$$
for $0\leq n \leq r$.  By the isomorphism theorem and by \cite[Proposition 3.5]{HM} we have
$$
\begin{array}{ccl}
  \overline{I}_{n}/\mathfrak{m}_{\ell-1}\overline{I}_{n}+a_{\ell}\overline{I}_{n-1} & \cong & I_{n}/(\mathfrak{m}I_{n}+a_{l}I_{n-1}+I_{n}\cap (a_{1},...,a_{\ell-1})) \vspace{0.3cm} \\

   & \cong &  I_{n}/(\mathfrak{m}I_{n}+(a_{1},...,a_{\ell})I_{n-1}).
\end{array}
$$
By the Lemma \ref{lemz31}, $\mu(\overline{I}_{n-1})-\mu(\overline{I}_{n-1})=\sum_{i=0}^{\ell}(-1)^{i}\binom{\ell}{i} \mu(I_{n-i})$. Then \textbf{(2)} follows.
\end{proof}

The Corollary below may be also view as a particular case of \cite[Theorem 2.3]{JN}.

\begin{cor}
Let $\mathfrak{F}$ be a Hilbert filtration. If $G(\mathfrak{F})$ is Cohen-Macaulay then $\emph{reg} \ G(\mathfrak{F}) \leq \emph{reg} \ F(\mathfrak{F})$. If  in addition $F(\mathfrak{F})$ is Cohen-Macaulay, then $\emph{reg} \ G(\mathfrak{F})=r=\emph{reg} \ F(\mathfrak{F})$.
\end{cor}

\begin{proof}
For the first part we use \cite[Theorem 4.6]{LP}. For the second part is enough observe \cite[Proposition 3.6]{HZ} for the first equality and the Proposition \ref{regf} for the second equality.
\end{proof}

If the rings $G(\mathfrak{F})$ and $F(\mathfrak{F})$ are Cohen-Macaulay, we see that even for Hilbert filtration, the regularities of $\omega_{G(\mathfrak{F})}$ and $\omega_{F(\mathfrak{F})}$ are equal. In addition, if they are also Gorenstein rings, their regularity are equal to the dimension of the base ring.

\begin{cor}
If the rings $G(\mathfrak{F})$ and $F(\mathfrak{F})$ are Cohen-Macaulay, one has $\emph{reg} \ \omega_{G(\mathfrak{F})}=\emph{reg} \ \omega_{F(\mathfrak{F})}$. Furthermore, if either $G(\mathfrak{F})$ or $F(\mathfrak{F})$ is Gorenstein, $\emph{reg} \ \omega_{G(\mathfrak{F})}=\emph{reg} \ \omega_{F(\mathfrak{F})}=d$.
\end{cor}

\begin{proof}
The case $d=0$ follows of some arguments in the proof for positive dimension. Assume $d>0$. By \cite[Proposition 2.2]{JV} and hypothesis, we get a minimal reduction $J$ of $\mathfrak{F}$ minimally generated by $x_{1},...,x_{d}$ such that $x_{1}^{*},...,x_{d}^{*}\in G(\mathfrak{F})$ and $x_{1}^{o},...,x_{d}^{o}\in F(\mathfrak{F})$ are regular sequences. Then
$$
F(\mathfrak{F})/(x_{1}^{o},...,x_{d}^{o}) \cong F(\mathfrak{F}/(x_{1},...,x_{d})) \  \text{and} \ G(\mathfrak{F})/(x_{1}^{*},...,x_{d}^{*}) \cong G(\mathfrak{F}/(x_{1},...,x_{d})).
$$
Then by \cite[Corollary 3.6.14]{BH},
$$
\text{reg} \ \omega_{F(\mathfrak{F})}= \text{reg}(\omega_{F(\mathfrak{F})}/J^{o}\omega_{F(\mathfrak{F})})= \text{reg}\omega_{F(\mathfrak{F}/J)}+d
$$
and the same for $G(\mathfrak{F})$, that is, $\text{reg} \hspace{0.1cm} \omega_{G(\mathfrak{F})}=\text{reg}\omega_{G(\mathfrak{F}/J)}+d$.

Since $\dim F(\mathfrak{F}/J)=0= \dim G(\mathfrak{F}/J)$,
$\text{reg} \hspace{0.05cm} \omega_{F(\mathfrak{F}/J)}=a(\omega_{F(\mathfrak{F}/J)})$ and $\text{reg} \hspace{0.05cm} \omega_{G(\mathfrak{F}/J)}=a(\omega_{G(\mathfrak{F}/J)})$.

By applying the Proposition \ref{prop3.1}\textbf{(1)} to $\mathfrak{F}/J$ we have
$$
\omega_{F(\mathfrak{F}/J)}\cong \oplus_{n \in \mathbb{Z}}(0:_{[\omega_{G(\mathfrak{F}/J)}]_{n}}\mathfrak{m}).
$$
Then by (\ref{extessencial}), we have $\omega_{[F(\mathfrak{F}/J)]_{n}} \neq 0 \Leftrightarrow \omega_{[G(\mathfrak{F}/J)]_{n}}\neq 0$ for all $n$. Hence $a(\omega_{[F(\mathfrak{F}/J)]_{n}})=a(\omega_{[G(\mathfrak{F}/J)]_{n}})$ and by the equalities above, $\text{reg} \ \omega_{G(\mathfrak{F})}= \text{reg} \ \omega_{F(\mathfrak{F})}$.

Furthermore if for example $G(\mathfrak{F})$ is Gorenstein, by \cite[Proposition 3.6.11]{BH}, $\omega_{G(\mathfrak{F})}\cong G(\mathfrak{F})(r-d)$. Then
$$
\text{reg} \ \omega_{G(\mathfrak{F})}= \text{reg} \ G(\mathfrak{F})(r-d)= \text{reg} \ G(\mathfrak{F})-r+d.
$$
Since $G(\mathfrak{F})$ is Cohen-Macaulay, by \cite[Proposition 3.6]{HZ}, $\text{reg} \ G(\mathfrak{F})=r$. Hence we got $\text{reg} \ \omega_{G(\mathfrak{F})}=d$. The result follows similarly if we suppose $F(\mathfrak{F})$ Gorenstein.

\end{proof}

In the Corollary below we see other descriptions (in terms of annihilators) of the canonical module of the Fiber cone and the canonical module of the associated graded ring.

\begin{cor}
Let $(A,\mathfrak{m})$ be a Gorenstein Noetherian local ring such that $d>0$ and $\mathfrak{F}$ a Hilbert filtration. Suppose that $G(\mathfrak{F})$ is Cohen-Macaulay. Then
$$\omega_{G(\mathfrak{F})} = \bigoplus_{n\in \mathbb{Z} }\frac{(J^{n+r-d}:I_{r})}{(J^{n+r-d+1}:I_{r})}$$ and $$\omega_{F(\mathfrak{F})} = \bigoplus_{n\in \mathbb{Z} }\frac{(J^{n+r-d+1}:\mathfrak{m}I_{r})\cap (J^{n+r-d}:I_{r})}{(J^{n+r-d+1}:I_{r})}.$$ Furthermore, $F(\mathfrak{F})$ is Gorenstein if and only if
$$
\frac{ (J^{n+1}:\mathfrak{m}I_{r})\cap (J^{n}: I_{r}) }{(J^{n+1}:I_{r}) }\cong \frac{I_{n}}{\mathfrak{m}I_{n}},
$$
for any $n\in \mathbb{Z}^{+}$.

\end{cor}

\begin{proof}
Due to hypothesis of this Corollary, the hypothesis from \cite{HKU2} are also satisfied and then
$$
\omega_{R^{'}(\mathfrak{F})}=\bigoplus_{n\in \mathbb{Z}}(J^{n+r}:I_{r})t^{n+d-1}.
$$
By \cite[Corollary 3.6.14]{BH},
$$
\omega_{G(\mathfrak{F})}=\omega_{R^{'}(\mathfrak{F})/(t^{-1})R^{'}(\mathfrak{F})}\cong (\omega_{R^{'}(\mathfrak{F})}/t^{-1}\omega_{R^{'}(\mathfrak{F})})(-1).
$$
Hence $\omega_{G(\mathfrak{F})}= \bigoplus_{n\in \mathbb{Z} }\frac{(J^{n+r-d}:I_{r})}{(J^{n+r-d+1}:I_{r})}.$
Next we use the Proposition \ref{prop3.1}\text{(1)} to get
$$
\begin{array}{ccl}
  \omega_{G(\mathfrak{F})} & = & \bigoplus_{n \in \mathbb{Z} } (0:_{[\omega_{G(\mathfrak{F})}]_{n}}\mathfrak{m}) \vspace{0.3cm} \\
   & = & \bigoplus_{n\in \mathbb{Z}}  \frac{(J^{n+r-d+1}:\mathfrak{m}I_{r})\cap (J^{n+r-d}:I_{r}) }{(J^{n+r-d+1}:I_{r})}
\end{array}
$$
Now suppose $G(\mathfrak{F})$ is Gorenstein. Then
$$
\omega_{F(\mathfrak{F})}\cong F(\mathfrak{F})(r-d)=  \bigoplus_{n\in \mathbb{Z}} \frac{ I_{n+r-d} } { \mathfrak{m}I_{n+r-d}. }
$$
Comparing both the graded isomorphism above, we have
$$
\frac{(J^{n+1}:\mathfrak{m}I_{r})\cap (J^{n}:I_{r}) }{(J^{n+1}:I_{r})} \cong \frac{I_{n}}{\mathfrak{m}I_{n}},
$$
for $n\in \mathbb{Z}^{+}$. The converse follows easily.

\end{proof}

The next Theorem gives a characterization for the fiber cone to be Gorenstein ring. We make a change in the conditions involving the lengths used in \cite{JN}.
The proof is analogous to the $I$-adic case.
\begin{thm}\label{teogor}
Let $(A,\mathfrak{m})$ be a Noetherian local ring and $J$ a minimal reduction of a Hilbert filtration $\mathfrak{F}$ with reduction number $r$. Suppose $G(\mathfrak{F})$ is Gorenstein and $F(\mathfrak{F})$ is Cohen-Macaulay. Therefore $F(\mathfrak{F})$ is Gorenstein if and only if
$$
\lambda\left( \frac{((I_{n+1}+J):\mathfrak{m})\cap I_{n} }{I_{n+1}+JI_{n-1} } \right)   =  \sum_{i=0}^{\ell}(-1)^{i}\binom{\ell}{i} \mu(I_{n-i})   \  \text{and} \  \lambda\left( \frac{(I_{1}:\mathfrak{m}) }{I_{1} } \right)=1,$$
for $1\leq n \leq r.$ If in addition we suppose that $A/I_{1}$ is Gorenstein we may delete the second equality of lengths.

\end{thm}

\begin{proof}

Since $G(\mathfrak{F})$ and $F(\mathfrak{F})$ are Cohen-Macaulay, by \cite[Proposition 2.2]{JV} and an argument similar to \cite[Theorem 8]{P}, we get a minimal reduction $J$ generated by elements whose the corresponding images form a regular sequence in  $G(\mathfrak{F})$ and $F(\mathfrak{F})$.
By applying the Proposition \ref{prop3.1} to $\mathfrak{F}/J$ and by using the fact of $\dim F(\mathfrak{F}/J) = 0$,
$$
\omega_{F(\mathfrak{F}/J)}  = \bigoplus_{n\in \mathbb{Z}} \left[ \frac{((I_{n+r+1}+J):\mathfrak{m})\cap I_{n+r} + J}{I_{n+r+1}+J} \right].
$$

By the isomorphism theorem and the fact that $J\cap I_{n}=JI_{n-1}$ for all $n$, we have
\begin{equation}\label{iso}
[\omega_{F(\mathfrak{F}/J)}]_{n-r}\cong \frac{((I_{n+1}+J):\mathfrak{m})\cap I_{n}}{I_{n+1}+JI_{n-1}}.
\end{equation}

If $F(\mathfrak{F})$ is Gorenstein, by Lemma \ref{lemgor}, it follows that $F(\mathfrak{F}/J)$ is also Gorenstein, and then $\omega_{F(\mathfrak{F}/J)}=F(\mathfrak{F}/J)(r)$. It is easy to check that $[F(\mathfrak{F}/J)(r)]_{n-r}=I_{n}/(\mathfrak{m}I_{n}+JI_{n-1})$.
Then by applying $\lambda$ at (\ref{iso}) and at the last equality, it gives
$$
\lambda\left( \frac{((I_{n+1}+J):\mathfrak{m})\cap I_{n} }{I_{n+1}+JI_{n-1} } \right)   =\lambda\left(\frac{I_{n}}{\mathfrak{m}I_{n}+JI_{n-1}}   \right),
$$
for all $n$.
By the Proposition \ref{regf}, we have the result desired.

For the converse, we again use the Proposition \ref{regf} to obtain that
$
\lambda ( ( (I_{n+1}+J):\mathfrak{m})\cap I_{n} ) /  ( I_{n+1}+JI_{n-1} )  )   = \lambda ( I_{n} / (\mathfrak{m}I_{n}+JI_{n-1}) )   ,
$
for $0\leq  n \leq r$. Of course this equality is true for all $n$ since $J$ is a reduction of $\mathfrak{F}$. Hence
by (\ref{iso}), $\lambda([\omega_{F(\mathfrak{F}/J)}]_{n})=\lambda ([F(\mathfrak{F}/J)(r)]_{n})$ for all $n$ and then
$\lambda(\omega_{F(\mathfrak{F}/J)})= \lambda(F(\mathfrak{F}/J)(r))= \lambda(F(\mathfrak{F}/J))$. Let $\mu=\mu(\omega_{F(\mathfrak{F}/J)})$ be the minimal number of generators of $\omega_{F(\mathfrak{F}/J)}$. We may construct a natural surjective map $\psi: H \rightarrow  \omega_{F(\mathfrak{F}/J)}$ such that $N$ is a free graded $F(\mathfrak{F}/J)$-module of rank $\mu$. By \cite[Corollary 36.11]{HIO}, $\omega_{F(\mathfrak{F}/J)}$ has finite injective dimension which is equal to $\text{depth} \ F(\mathfrak{F}/J)=\dim F(\mathfrak{F}/J)=0$. This way $\omega_{F(\mathfrak{F}/J)}$ is an injective module and consequently $\text{Hom}_{F(\mathfrak{F}/J)}(-, \omega_{F(\mathfrak{F}/J)})$ is an exact functor. Then we have the following exact sequence
$$
\psi^{*}: \text{Hom}_{F(\mathfrak{F}/J)}( \omega_{F(\mathfrak{F}/J)} , \omega_{F(\mathfrak{F}/J)})  \rightarrow \text{Hom}_{F(\mathfrak{F}/J)}(H, \omega_{F(\mathfrak{F}/J)}).
$$
By \cite[Theorem 13.3.4(ii)]{BS},
$$
\text{Hom}_{F(\mathfrak{F}/J)}( \omega_{F(\mathfrak{F}/J)} , \omega_{F(\mathfrak{F}/J)})\cong F(\mathfrak{F}/J).
$$
Naturally
$$
\text{Hom}_{F(\mathfrak{F}/J)}( H , \omega_{F(\mathfrak{F}/J)})\cong \bigoplus_{\mu} \omega_{F(\mathfrak{F}/J)}.
$$
We have gotten a surjective map $F(\mathfrak{F}/J) \rightarrow \bigoplus_{\mu} \omega_{F(\mathfrak{F}/J)}$. Hence
$$
\mu \cdot \lambda( \omega_{F(\mathfrak{F}/J)} )  = \lambda(\bigoplus_{\mu} \omega_{F(\mathfrak{F}/J)}) \leq \lambda(F(\mathfrak{F}/J)).
$$
This implies that $\mu=1$. By \cite[Proposition 3.6.11]{BH}, $F(\mathfrak{F}/J)$ is Gorenstein. And by Lemma \ref{lemgor} the result follows.
\end{proof}

Through the last Theorem in the adic-filtration case \cite{JN}, Jayanthan and Nanduri find some results about the quotient ring $A/I$. Similarly one can obtain properties for $A/I_{1}$. In \cite[Theorem 4.5.7]{BH}, it is proved that if $G(\mathfrak{F})$ is Gorenstein, $A$ so is. But in general, it does not occur for fiber cone.
In the Corollary \ref{corA/I}, we see that if $\mathfrak{F}$ is a Hilbert filtration and both $G(\mathfrak{F})$ and $F(\mathfrak{F})$ are Gorenstein, then $A/I_{1}$ is Gorenstein.

\begin{cor}\label{corA/I}
Let $(A,\mathfrak{m})$ be a Noetherian local ring and $\mathfrak{F}$ a Hilbert filtration. If both $G(\mathfrak{F})$ and $F(\mathfrak{F})$ are Gorenstein, then $A/I_{1}$ so is.
\end{cor}

\begin{proof}
By Theorem \ref{teogor}, for $n=0$, we have $\lambda((I_{1}:\mathfrak{m})/I_{1})=\lambda(A/\mathfrak{m})=1$. It implies that $(I_{1}:\mathfrak{m})/I_{1}\cong A/\mathfrak{m} $. By \cite[Theorem 11.12]{ILL}, the result is done.
\end{proof}

\begin{cor}
Let $(A,\mathfrak{m})$ be a Noetherian local ring and $\mathfrak{F}$ a Hilbert filtration. Assume that $G(\mathfrak{F})$ is Gorenstein, $F(\mathfrak{F})$ is Cohen-Macaulay and $A/I_{1}$ is Gorenstein. If we suppose that $(I_{n+r-d+1}:\mathfrak{m})\cap \mathfrak{m}I_{n+r-d}= I_{n+r-d+1} $ for all $n \geq d-r+1$,
$\omega_{F(\mathfrak{F})}$ is a submodule of $F(\mathfrak{F})$.
\end{cor}

\begin{proof}
By hypothesis, $I_{n+1}\subseteq  \mathfrak{m}I_{n}$ for any $n$. Then there is a natural map
$$
\psi_{n}:\frac{(I_{n+r-d+1}:\mathfrak{m})\cap I_{n+r-d}}{I_{n+r-d+1}}\rightarrow \frac{I_{n+r-d}}{\mathfrak{m}I_{n+r-d}}
$$
for $n \geq d-r$ where
$$
\text{Ker} \ \psi_{n} = \frac{(I_{n+r-d+1}:\mathfrak{m})\cap \mathfrak{m}I_{n+r-d}}{I_{n+r-d+1}}.
$$
By (\ref{n>d-r}) or simply Proposition \ref{prop3.1} it gives a natural $F(\mathfrak{F})$-linear map
$$
\psi : \omega_{F(\mathfrak{F})} \rightarrow F(\mathfrak{F}).
$$
Since $A/I_{1}$ is Gorenstein, by using \cite[Theorem 11.12 (3)]{ILL}, $\text{Soc}(A/I_{1})=(I_{1}:\mathfrak{m})/I_{1} \cong A/\mathfrak{m}$, which gives $\text{Ker} \ \psi_{d-r}=0$. Then by last hypothesis,
$$
\text{ker} \ \psi =  \bigoplus_{n \geq d-r} \frac{(I_{n+r-d+1}:\mathfrak{m})\cap \mathfrak{m}I_{n+r-d}}{I_{n+r-d+1}} =0.
$$
The result is concluded.
\end{proof}

In \cite{JN}, it is showed that $e_{0}(\omega_{F(\mathfrak{F})})< e_{0}(\omega_{G(\mathfrak{F})})$ unless $I=\mathfrak{m}$. For Hilbert filtration the result is similar.

\begin{prop}
Let $(A,\mathfrak{m})$ be a Noetherian local ring and $\mathfrak{F}$ a Hilbert filtration such that $G(\mathfrak{F})$ and $F(\mathfrak{F})$ are Cohen-Macaulay. Then
$$
e_{0}(\omega_{F(\mathfrak{F})}) \leq e_{0}(\omega_{G(\mathfrak{F})}).
$$
In addition, if $G(\mathfrak{F})$ is Gorenstein, we have the equality if and only if $I_{1}=\mathfrak{m}$.
\end{prop}

\begin{proof}
Proposition \ref{prop3.1} gives us that $\omega_{F(\mathfrak{F})}  \subseteq \omega_{G(\mathfrak{F})}$. By additivity of $e_{0}$,
$$
e_{0}(\omega_{F(\mathfrak{F})}) \leq e_{0}( \omega_{G(\mathfrak{F})} ).
$$
By hypothesis and the properties of canonical module we may find a minimal reduction $J=x_{1},...,x_{d}$ of $\mathfrak{F}$ such that $x_{1}^{*},...,x_{d}^{*}\in G(\mathfrak{F})$ are regular sequence and $\omega_{G(\mathfrak{F})}$-regular sequence and $x_{1}^{o},...,x_{d}^{o}\in F(\mathfrak{F})$ are regular sequence and $\omega_{F(\mathfrak{F})}$-regular sequence. By \cite[Remark 4.1.11]{HZ} and using that $\dim F(\mathfrak{F})=0$ and $\dim G(\mathfrak{F})=0$ we have
$$
e_{0}(\omega_{G(\mathfrak{F})})=e_{0}(\omega_{G(\mathfrak{F}/J)})=\lambda ( G(\mathfrak{F}/J) )
$$
and
$$
e_{0}(\omega_{F(\mathfrak{F})})=e_{0}(\omega_{F(\mathfrak{F}/J)})=\lambda ( F(\mathfrak{F}/J) ).
$$

If  $e_{0}(\omega_{G(\mathfrak{F})}) = e_{0}(\omega_{F(\mathfrak{F})})$ , then $\lambda ( F(\mathfrak{F}/J) )= \lambda ( G(\mathfrak{F}/J) )$. By Proposition \ref{prop3.1} we have $\omega_{F(\mathfrak{F}/J)} \subseteq \omega_{G(\mathfrak{F}/J)}$ and hence $\omega_{F(\mathfrak{F}/J)} = \omega_{G(\mathfrak{F}/J)}$.
Since $G(\mathfrak{F}/J)$ is Gorenstein and by using the isomorphism (\ref{iso}) we get $(I_{1}: \mathfrak{m})/I_{1} = A/I_{1}$ which means $I_{1}=\mathfrak{m}$.

Due to the hypothesis $I_{n+1} \subseteq  \mathfrak{m}I_{n}$, it is easy to show that $\mathfrak{F} =(\mathfrak{m}^{n})_{n \geq 0}.$ Then the converse follows trivially.

\end{proof}

\begin{agra}
We would like to thank the professors J. V. Jayanthan and R. Nanduri for the clarifications.
\end{agra}

\textbf{Department of Mathematics, Institute of Mathematics and Computer Science, ICMC, University of S\~{a}o Paulo, BRAZIL.
}

\emph{E-mail address}: apoliano27@gmail.com

\hspace{1cm}

\textbf{Department of Mathematics, Institute of Mathematics and Computer Science, ICMC, University of S\~{a}o Paulo, BRAZIL.
}

\emph{E-mail address}: vhjperez@icmc.usp.br

\end{document}